\documentclass[11pt]{article}
 
\usepackage{fullpage,amsfonts,amsmath,amsthm,amssymb,mathrsfs,graphicx,epstopdf}
\usepackage[top=2.54cm, bottom=2.54cm, left=2.54 cm, right=2.54 cm]{geometry}

\usepackage{amsfonts,amsmath,latexsym,hyperref,graphicx,amssymb,amsthm,url,lineno,float,color}
\usepackage[]{algorithm2e}
\usepackage[abbrev,msc-links]{amsrefs}  
\usepackage{subcaption}
\usepackage{bbm}
\usepackage{enumitem}

 \newcommand{\lab}[1]{\label{#1}\ \ {\bf[#1]}\ }

 \usepackage[usenames,dvipsnames]{xcolor}
\newcommand{\remove}[1]{}
\newcommand\eqn[1]{(\ref{#1})}
\newcommand{\be}{\begin{equation}}
\newcommand{\bel}[1]{\begin{equation}\lab{#1}\ }
\newcommand{\ee}{\end{equation}}
\newcommand{\bea}{\begin{eqnarray}}
\newcommand{\eea}{\end{eqnarray}}
\newcommand{\bean}{\begin{eqnarray*}}
\newcommand{\eean}{\end{eqnarray*}}

\newtheorem{thm}{Theorem}
\newtheorem{cor}[thm]{Corollary}

\newtheorem{lemma}[thm]{Lemma}
\newtheorem{definition}[thm]{Definition}

\def\proof{\noindent{\bf Proof.\ }}
\def\qed{~~\vrule height8pt width4pt depth0pt}
\def\ss{{\smallskip}}

\newcommand{\wt}{\widetilde}

\def\UB{{\overline f}}
\def\LB{{\underline b}}

\def\OV{{\overline V}}
\def\Gen{{\tt INC-GEN}}
\def\reg{{\tt INC-REG}}
\def\powerlaw{{\tt INC-POWERLAW}}
\def\bipartite{{\tt INC-BIPARTITE}}

\def\F{{\mathcal F}}
\def\G{{\mathcal G}}

\def\S{{\mathcal S}}

\def\pr{{\mathbb P}}

\def\bfd{{\bf d}}

\def\bfm{{\bf m}}

\def\bfs{{\bf s}}
\def\bft{{\bf t}}

\def\UB{{\overline f}}
\def\LB{{\underline b}}
\def\NoLoops{{\tt NoLoops}}
\def\NoDoubles{{\tt NoDoubles}}
\newcommand{\Unroot}{{\tt Loosen}}
\newcommand{\Deanchor}{{\tt Relax}}
\providecommand{\keywords}[1]{\hspace{0.5in}\textbf{{Keywords: }} #1}

\date{}
\title{Fast uniform generation of random graphs with given degree sequences\footnote{An extended
abstract of this paper appeared in the proceeding of FOCS2019}}{\pagestyle{empty}}
\author{Andrii Arman\and Pu Gao\and Nicholas Wormald}
\author{
Andrii Arman\\
School of Mathematics\\
Monash University\\
andrii.arman@monash.edu
\and
Pu Gao\thanks{Research supported by ARC DP160100835 and NSERC.}\\
Department of Combinatorics and Optimization\\
University of Waterloo\\
pu.gao@uwaterloo.ca
\and
Nicholas Wormald\thanks{Research supported by  ARC DP160100835.}\\
School of Mathematics\\
Monash University\\
nick.wormald@monash.edu }

\begin{document}
\maketitle

\begin{abstract}

In this paper we provide an algorithm that generates a graph with given degree sequence uniformly at random. Provided that $\Delta^4=O(m)$, where $\Delta$ is the maximal degree and $m$ is the number of edges,
 the algorithm runs in expected time $O(m)$. Our algorithm significantly improves the previously most efficient uniform sampler, which runs in expected time $O(m^2\Delta^2)$ for the same family of degree sequences. Our method uses a novel ingredient which progressively relaxes restrictions on an object being generated uniformly at random, and we use this to give fast algorithms for uniform sampling of graphs with other degree sequences as well. Using the same method, we also  obtain algorithms with expected run time which is (i) linear for power-law degree sequences in cases where the previous best was $O(n^{4.081})$, and  (ii) $O(nd+d^4)$ for $d$-regular graphs when $d=o(\sqrt n)$, where the previous best was $O(nd^3)$.
\end{abstract}
\keywords{randomised generation algorithms, random graphs, rejection sampling}
\section{Introduction}\label{s:intro}

Sampling discrete objects from a specified probability distribution is a classical problem in computer science, both in theory and for practical applications. Uniform generation of random graphs with a specified degree sequence is one such problem  that has frequently been studied. In this paper we consider only the task of generating   {\em simple} graphs, i.e.\ graphs with no loops or multiple edges. An early algorithm was given by Tinhofer~\cite{tinhofer79}, but with unknown run time. A simple rejection-based uniform generation algorithm is usually implicit for asymptotically enumerating graphs with a specified degree sequence, for example in the papers of B{\'e}k{\'e}ssy, B{\'e}k{\'e}ssy and Koml\'{o}s ~\cite{bekessy1972},  Bender and Canfield~\cite{bender1978} and Bollob{\'a}s~\cite{bollobas1980}. The run time of this algorithm is linear in $n$ but exponential in the square of the average degree. Hence it only works in practice when degrees are small. 

A big increase in the permitted degrees of the vertices  was achieved by McKay and Wormald~\cite{mckay90}, and around the same time  Jerrum and Sinclair~\cite{jerrum90} found an approximately uniform sampler using Markov Chain Monte Carlo (MCMC) methods.
 McKay and Wormald used the configuration model introduced in~\cite{bollobas1980} to generate a random (but not uniformly random) multigraph with a given degree sequence. Instead of repeatedly rejecting until finding a simple graph, McKay and Wormald used a switching operation to switch away multiple edges, reaching a simple graph in the end. The algorithm is rather efficient when the degrees are not too large. In particular, for $d$-regular graphs it runs  in expected time  $O(d^3n)$ when $d=O(n^{1/3})$. (Here and in the following we assume $n$ is the number of vertices.)  Jerrum and Sinclair's Markov chain mixes in time polynomial  in  $n$ provided that  the degree sequence satisfies a condition phrased in terms of the numbers of graphs of given degree sequences. In particular, the mixing time is polynomial in the $d$-regular case for any function $d=d(n)$. These two benchmark research papers led the study into two different research lines. More switching-based algorithms for exactly uniform generation were given which deal with new degree sequences permitting vertices of higher degrees. The regular case was treated by Gao and Wormald~\cite{gao17} for $d=o(\sqrt n)$ with time complexity again $O(d^3n)$, and very non-regular but still quite sparse degree sequences (such as power law)~\cite{gao18} were considered by the same authors. 
Various MCMC-based algorithms have been investigated for generating the graphs with distribution that is only approximately uniform, e.g. algorithms by Cooper, Dyer and Greenhill~\cite{cooper07}, Greenhill~\cite{greenhill14}, Kannan, Tetali and Vempala~\cite{kannan99}.
These algorithms can cope with a much bigger family of degree sequences than the switching-based algorithms. That these do not produce the exactly uniform distribution might be irrelevant for practical purposes, if it were not for the fact that  the theoretically provable mixing bounds are too big. For instance, the mixing time was bounded by $d^{16}n^9\log (n/\epsilon)$ in~\cite{cooper07} in the regular case. We note that there have also been switching-based approximate samplers that run fast (in linear or sub-quadratic time), for instance see paper of Bayati, Kim and Saberi~\cite{bayati10}, Kim and Vu~\cite{kim03}, Steger and Wormald~\cite{steger99} and Zhao~\cite{zhao13}. For those algorithms,  the  bounds on error in the output distribution are functions of $n$ which tend to 0 as $n$ grows, but  cannot be reduced for any particular $n$ by running the algorithm longer. In this way they differ from the  MCMC-based algorithms, which are fully-polynomial almost uniform generators in the sense of~\cite{jerrum90}.

The goal of this paper is to introduce a new technique for exactly uniform generation. Using it to modify  switching-based algorithms, we can obtain vastly reduced run times, specifically, we aim for linear-time algorithms. In the context of generating a random graph, this should be linear in the number of edges, i.e.\ $O(M)$, where we use $M$ to denote the sum of the degrees in the graph. In particular, we obtain a linear-time  algorithm that works for the same family of degree sequences as the $O(M^2\Delta^2)$ algorithm in~\cite{mckay90}.  We first review the salient features of the latter  algorithm.

The algorithm first generates an initial random multigraph in expected time that is linear in $M$. (We describe the algorithm here in terms of multigraphs, though it is presented in~\cite{mckay90} in terms of pairings occuring in the above-mentioned configuration model.) The initial multigraph contains no loops of multiplicity at least two, no multiple edges of multiplicity at least three, and has a sublinear number of loops and double edges. The algorithm then uses an operation called $d$-switching to sequentially ``switch away" all the double edges (loops are treated similarly so we ignore them at present). Provided that a multigraph $G$ was uniform in the class of graphs with $m_2$ double edges, the result of applying a random $d$-switching to $G$ is a random multigraph $G'$ that is slightly non-uniformly distributed in a class of multigraphs with $m_2-1$ double edges. The following rejection scheme is used to equalise probabilities. Let $f_{d}(\wt{G})$ be the number of ways that a $d$-switching can be performed on $\wt{G}$ and $b_{d}(\wt{G})$ be the number of $d$-switchings that can create $\wt{G}$. Assume that $\UB_d(m)$ and $\LB_d(m)$ are uniform upper and lower bounds for $f_d(\wt{G})$ and $b_d(\wt{G})$ respectively over all multigraphs with $m$ double edges. If a switching that converts some multigraph  $G$ to a multigraph $G'$ is selected by the algorithm, then the switching is accepted with probability $f_d(G)\LB_d(m_2-1)\slash \UB_d(m_2)b_d(G')$, and rejected otherwise. If the switching is accepted, it is applied to the multigraph, whereas rejection requires re-starting the algorithm from scratch. Computing $b_d(G')$ takes $O(M^2\Delta^2)$ time, which dominates the time complexity of~\cite{mckay90}.

The algorithm presented in this paper is obtained from the algorithm in~\cite{mckay90} by modifying the time-consuming rejection scheme. 
First, it was observed in~\cite{mckay90} that the rejection can be separated into two distinct steps, which are given the explicit names f- and b-rejection in~\cite{gao17}. The f-rejection step rejects the selected switching with probability $1-f_d(G)\slash \overline{f}_{d}(m_2)$, and the  b-rejection step rejects it with probability $1- \LB_d(m_2-2)/b_{d}(G^\prime)$.  It is easy to see that  the overall probability of accepting the switching is the same as specified originally above.
By a slick observation, there is essentially no computation cost for computing the probability of f-rejection. (See the explanations in Section~\ref{sec:time}). The modification in the present paper is to further separate b-rejections into a sequence of sub-rejections by a scheme we will call \emph{incremental relaxation}. This scheme  will still maintain uniformity of the multigraphs created.

The basic idea of incremental relaxation, as used in the present paper, can be described as follows.  Let $H$ be a (small) graph with each edge designated as positive or negative. We say that an {\em  $H$-anchoring} of a graph $G$ is an injection   $Q:V(H)\to V(G)$   that maps every positive edge of $H$ to an edge of $G$, and every negative edge to a non-edge of $G$. (This is a generalisation of rooting at a subgraph, which usually corresponds to the case that $H$ has positive edges only.)  

Now assume that an $H$-anchored graph  $(G,  Q)$ is chosen u.a.r., i.e.\ each such ordered pair with $G$   in some given set $\cal O$, and $Q$, an $H$-anchoring of $G$, is equally likely.   We can convert this to a random graph $G\in \cal O$   by finding the number $b(G)$ of $H$-anchorings of $G$, and accepting $G$ with probability  $\LB({\cal O})/ b(G)$ where $\LB({\cal O})$ is a lower bound on the number of $H$-anchorings of any element $G'\in {\cal O}$. 
However, computing $b(G)$   corresponds to  computing $b_d(G')$ as 
 described above and can be time-consuming.  The key idea of our new method is that we   {\em incrementally} relax the constraints imposed on $G$ by $Q$,   so that   rejection is split  into a sequence of sub-rejections. Set
$\emptyset=V_0\subseteq V_1\subseteq \cdots \subseteq V_k=V(H)$ and let $Q_i$ denote the restriction of $Q$ to $V_i$. With this definition, for each $i$, $Q_i$ is an $H[V_i]$-anchoring of $G$. Thus $Q_i$ determines some subset (increasing with $i$) of the constraints on $G$ corresponding to the edges of $H$, and  given that $(G,  Q_{i})$ is uniformly random, we can obtain a uniformly random anchoring  $(G,  Q_{i-1})$   by  applying a similar rejection strategy, but using only the number $b(G,Q_{i-1})$  of ways that $ Q_{i-1}$ can be extended to an $H[V_i]$-anchoring of $G$. 
This procedure of incremental relaxation of constraints can be highly advantageous if for each $i$, $b(G,Q_{i-1})$ can be computed much faster than $b(G)$. In this way, a sequence of uniformly random objects is obtained, involving anchorings at ever-smaller subgraphs of $H$, until the empty subgraph is reached, corresponding to obtaining $G$ u.a.r.

To see that  this idea applies to the problem at hand, we observe that
the existence of a  $d$-switching (defined in Section~\ref{subsec:nodoubles})  from $G$ to $G'$  forces $G'$ to include a set $A$ of edges (the positive edges, forming two paths of length 2, in a copy of a certain graph $H$), and to exclude a set $B$  (the negative edges,  forming a matching, in $H$). So $G'$ comes accompanied by  an $H$-anchoring.(Refer to right side of Figure~\ref{f:d-switching} for a drawing of $H$.) To apply incremental relaxation we first compute the number of ways to complete such an anchoring given the first 2-path and use that to obtain a random 2-path-anchored graph, and then relax the 2-path anchoring in a similar manner. The details of applying this scheme to $d$-switchings are given in Section~\ref{subsec:nodoubles}.

   In Section~\ref{s:deanchor} we present the incremental relaxation technique in a more general setting, avoiding injections but instead employing more arbitrary sets of constraints. 
 We apply the incremental relaxation scheme in detail in the case $\Delta^4=O(M)$ (e.g.\ $d=O(n^{1/3})$ in the regular degree case) in Sections~\ref{sec:algorithm} --~\ref{sec:time}. The switchings we use are exactly the same as those in~\cite{mckay90}. When the incremental relaxation scheme is combined with the new techniques introduced in~\cites{gao17,gao18}, it allows us to obtain fast uniform samplers of graphs for the family of degree sequences permitted in~\cites{gao17,gao18}. In particular, we obtain  a linear-time algorithm to generate graphs with power-law degrees, and a sub-quadratic-time algorithm to generate $d$-regular graphs when $d=o(n^{1/2})$. We will discuss these algorithms in Sections~\ref{sec:regular} and~\ref{sec:powerlaw}. 


\section{Main results}
Let $\bfd=(d_1, \ldots, d_n)$ be specified where $M=\sum d_i$ is even. Let $\Delta=\max \{d_1, \ldots, d_n\}$ and for positive integers $j$ define $M_j=\sum_{i=1}^{n} d_i(d_i-1)\cdots(d_i-j+1)$. Note  that $M_j\leq \Delta^{j-1}M$ for all $j$.

 We say that $\bfd$ is \emph{graphical} if there exists a simple graph with degree sequence $\bfd$. For the rest of this paper we only consider graphical sequences $\bfd$. Our first result  is that our  algorithm \Gen\  uniformly generates a random graph with degree sequence $\bfd$ and runs in linear time provided that $\bfd$ is ``moderately sparse". The description of \Gen\ is given in Section~\ref{sec:algorithm}. The proof of the uniformity will be presented in Section~\ref{sec:uniformity}, and the time complexity is bounded in Section~\ref{sec:time}.
\begin{thm}\label{thm:main} Let $\bfd$ be a graphical sequence. Algorithm \Gen\ uniformly generates a random graph with degree sequence $\bfd$. 
If $\Delta^4=O(M)$ then the expected run time of \Gen\ is $O(M)$. The space complexity of \Gen\ is $O(n^2)$.
\end{thm}
Our second algorithm, \reg, described in Section~\ref{sec:regular}, is an almost-linear-time algorithm to generate random regular graphs. The run time is $O(dn+d^4)$ when $d=o(n^{1/2})$. This improves the $O(d^3n)$ run time of the uniform sampler in~\cite{gao17}.
\begin{thm}\label{thm:regular}
Algorithm \reg\ uniformly generates a random $d$-regular graph. 
If $d=o(n^{1/2})$ then the expected run time of \reg\ is $O(dn+d^4)$.
\end{thm}
Our third algorithm, \powerlaw, described in Section~\ref{sec:powerlaw}, is a linear-time algorithm to generate random graphs with a power-law degree sequence. A degree sequence $\bfd$ is said to be {\em power-law distribution-bounded} with parameter $\gamma>1$, if the minimum component in $\bfd$ is at least 1, and there is a constant $K>0$ independent of $n$ such that the number of components that are at least $i$ is at most $Kni^{1-\gamma}$ for all $i\ge 1$. Note that the family of power-law distribution-bounded degree sequences covers the family of degree sequences arising from $n$ i.i.d.\ copies of a  power-law  random variable. Uniform generation of graphs with power-law distribution-bounded degree sequences with parameter $\gamma>21/10+\sqrt{61}/10\approx 2.881024968$ was studied in~\cite{gao18}, where a uniform sampler was described with expected run time $O(n^{4.081})$. This was the first known uniform sampler for this family of degree sequences. With our new rejection scheme, we improve the time complexity to linear.
 
\begin{thm}\label{thm:powerlaw} Let $\bfd$ be a power-law distribution-bounded degree sequence with parameter $\gamma>21/10+\sqrt{61}/10\approx 2.881024968$.
Algorithm \powerlaw\ uniformly generates a random graph with degree sequence $\bfd$, and the expected run time  of \powerlaw\ is $O(n)$.
\end{thm}

Algorithms \Gen\ and \reg\ can easily be  modified if $\bfd$ represents a bipartite graph's  degree sequence. As an example, we present algorithm \bipartite\ in Section~\ref{sec:bipartite} as the bipartite version of \Gen. 
\begin{thm}\label{thm:bipartite} Algorithm \bipartite\ uniformly generates a random graph with bipartite degree sequence $\bfd=(\bfs,\bft)$. 
If $\Delta^4=O(M)$ then the expected run time of \bipartite\ is $O(M)$. The space complexity of \bipartite\ is $O(mn)$.
\end{thm}
\section{Uniform generation by incremental relaxation}\label{s:deanchor}

We provide here a general description of the relaxation procedure, so it can be applied in different setups. Let $\F$ and $k$ be given, where $\F$ is a finite set and $k$ is a positive integer. We are also given $S_i$, for $1 \le i \le  k$, where each $S_i$ is a multiset consisting of subsets of $\F$.  Let $\otimes$ denote the Cartesian product, and let $\F_k$ be any subset of $\F\times S_1\times\cdots \times S_k$ such that each $(G,C_1,\ldots, C_k)\in\F_k$ satisfies $G\in C_k\subseteq \cdots \subseteq C_1$.  Given $F=(G,C_1,\ldots, C_k) \in \F_k$, define $P_i(F)=(G,C_1,\ldots, C_i)$ for each $1\le i <k$. For each $i \in [k-1]$ set $\F_i=\{P_i(F) : F\in{\F_k} \}$ and set $\F_0=\F$.  

 For any $i \in [k]$ and $F:=(G,C_1,\ldots, C_i)\in \F_i$, define  $P(F)=(G,C_1,\ldots, C_{i-1})\in \F_{i-1}$; i.e.\ $P(F)$ is the  prefix of $F$.
  
Later in our applications of relaxation, we will let $\F$ be a set of multigraphs. Each element $F$ of $\F_i$ can be identified with a multigraph that contains a specified substructure (determined  by the $C_i$-s) on a specified set of vertices. In terms of the notation introduced in Section~\ref{s:intro}, elements of $\F_i$ will correspond to $H[V_i]$-anchorings of multigraphs for some graph $H$ and some sequence $\emptyset =V_0\subseteq V_1\subseteq \cdots \subseteq V_k=V(H)$. Permitting multiple copies of elements in $S_i$ is useful in the case where two distinct constraints may correspond to the same subset of $\F$. This happens in  our applications due to the symmetry of the substructures in $H$.

Next we define a procedure \Unroot, which takes an $F=(G,C_1, \ldots,C_i)\in \F_i$ as input, and outputs an $P(F)\in \F_{i-1}$ with a certain probability and otherwise `rejects' it and terminates. Our Relaxation Lemma (Lemma~\ref{lemma:unroot} below) shows that if $F$ is uniformly distributed in $\F_i$ then the output of \Unroot\ is uniformly distributed in $\F_{i-1}$.

For $0\leq i \leq k-1$ and $F\in \F_{i}$, let $b(F)$ be the number of ${F'}\in \F_{i+1}$ such that $P({F'})=F$. In other words, $b(F)$ is the number of ways to extend $F$ to an element of $\F_{i+1}$. Let $\LB(i)$ be a lower bound on 
$b(F)$ over all $F\in \F_i$, and assume that  for all $i\in [k-1]$, $\LB(i)>0$.
For $F\in \F_i$ with $i\ge 1$ we define the following procedure.

\begin{algorithm}[H]
{\bf{procedure}}
{\Unroot $(F)$:}\\
 Output $P(F)$ with probability $\LB(i-1)\slash b(P(F))$, and reject otherwise.
\end{algorithm}

\ss

Procedure {\Deanchor} is defined for $F=(G,C_1,\ldots,C_k) \in \F_k$. It repeatedly calls {\Unroot} until reaching a $G\in \F_0$. We say that procedure \Deanchor\ performs \emph{incremental relaxation} on $(G,C_1,\ldots,C_k)$.

\begin{algorithm}[H]
{\bf{procedure}}
{\Deanchor $(F)$:}\\
$i\gets k$;\\
\While{$i\ge 1$}{
$(G,C_{1}, \ldots, C_{i-1})={\Unroot}(G,C_{1}, \ldots, C_{i})$;\\
$i \gets i-1$.
}
Output $G$.
\end{algorithm}

\begin{lemma}[Relaxation Lemma]\label{lemma:unroot} Assume that $i\in [k]$ and $\LB(i-1)>0$. Provided that $F \in \F_i$ is chosen uniformly at random, the output of \Unroot$(F)$ is uniform in $\F_{i-1}$ assuming no rejection.
\end{lemma}
\begin{proof}
Let $p=\frac{1}{|\F_i |}$.  For any $F' \in \F_{i-1}$, the probability that \Unroot\ outputs $F'$ is equal to 
\[
\sum_{ F\in \F_{i} \; : \; P(F)=F'} \pr(A_{F}) \pr(\mbox{no rejection} \mid A_{F}),
\]
where $A_{F}$ denotes the event that the input of \Unroot\ is $F$. The second probability above is the conditional probability that no rejection occurs in \Unroot, given $A_{F}$. By our assumption, the first probability above is always equal to $p$. By the definition of \Unroot, the second probability above is equal to $\LB(i-1)/b(F')$.  By definition, $b(F')$ is exactly the number of $F\in \F_i$, such that $P(F)=F'$, so the sum has exactly $b(F')$ terms, each of which is equal to $p\LB(i-1)/b(F')$. Hence,
the probability for \Unroot\ to output $F'$ is equal to $p\LB(i-1)$, for every $F'\in \F_{i-1}$. \qed
\end{proof}

Recalling that $\F_0=\F$, the Relaxation Lemma immediately yields the following corollary for the uniformity of Procedure \Deanchor.

\begin{cor}\label{cor:de-anchor} Assume that for all $i \in [k]$, $\LB({i-1})>0$, and assume $F\in \F_{k}$ is chosen uniformly at random. Then the output of \Deanchor$ (F)$ is uniform in $\F$, if there is no rejection. 
\end{cor}
The description of \Deanchor\ as repeated calls of \Unroot\ is useful for analysing the algorithm, but for practical implementations we refer to the following corollary.
\begin{cor}\label{cor:Deanchor}
Procedure \Deanchor, when applied to $(G,C_1,\ldots,C_k)\in \F_k$, outputs $G$ with probability \\ $\prod_{i=0}^{k-1}\LB(i)/b(G,C_1, \ldots, C_i)$, and ends in rejection otherwise. 
\end{cor}
In practice, we predefine the numbers  $\LB(i)$. Once the numbers $b(G,C_1, \ldots, C_i)$ are computed, the b-rejection can be performed in one step using Corollary~\ref{cor:Deanchor}, and  there is no need to perform  \Deanchor\ with its  iterated calls to \Unroot. As mentioned in Section~\ref{s:intro}, these numbers can be much faster to compute than the number of $H$-anchorings of $G$, which would be required using the scheme in~\cite{mckay90}. We also reiterate that, unlike the scheme in~\cite{mckay90}, the rejection probability depends on the anchoring imposed by $C_k$, as well as $G$.

\section{Algorithm \Gen}
\label{sec:algorithm}

In this section we provide a description of \Gen. Let $\bfd$ be given. We will use the configuration model~\cite{bollobas1980} to generate a random pairing, defined as follows. For every $1\le i\le n$, represent vertex $v_i$ as a bin containing exactly $d_i$ points. Take a uniformly random perfect matching over the set of points in the $n$ bins. Call the resulting matching $P$ a {\em pairing} and call each edge in $P$ a {\em pair}. Finally identify the bins as vertices, and represent each pair in $P$ as an edge. This produces a multigraph from $P$, denoted by $G(P)$. If a set of pairs in $P$ form a multiple edge or loop in $G(P)$ then this set of pairs is called a multiple edge in $P$ as well, with the same multiplicity as it has  in $G(P)$. A loop is a pair with both ends contained in the same bin/vertex. If there is a set containing more than one pair with all ends contained in the same vertex, then this set of pairs form a multiple loop. We always use loop to refer to a single loop with multiplicity equal to one. We call a multiple edge with multiplicity 2 or 3 a double or triple edge respectively. Let $\Phi(\bfd)$ denote the set of all pairings with degree sequence $\bfd$. Recall that $\Delta=\max_{i\in[n]} d_i$, $M=\sum_{i=1}^{n}d_i$ and $M_2=\sum_{i=1}^{n}d_i(d_i-1)$. 

 If $22\Delta^3<M_2$ define 
\be\label{bounds}
 B_1=\frac{M_2}{M}, \quad B_2=\left(\frac{M_2}{M}\right)^2,
\ee
and define $B_1=B_2=0$ otherwise.  The consideration of two cases is needed to ensure that certain  parameters defined in Section~\ref{sec:parloops} and Section~\ref{sec:pardoubles} are positive, and thereby to ensure that the algorithm has finite expected runtime.

Let $\Phi_0$ denote the set of pairings in $\Phi(\bfd)$ where there are no multiple edges with multiplicity at least 3, and no multiple loops with multiplicity at least 2, and the number of loops and double edges are at most $B_1$ and $B_2$ respectively. The following result is essentially contained in~\cite{mckay90} so we only give a brief description of the proof.

\begin{lemma}\label{lem:initial} Let $\bfd$ be a graphical degree sequence with $\Delta^4=O(M)$ and $P$ be a uniformly random pairing in $\Phi(\bfd)$.
Then there exists a constant $0<c<1$ such that $\pr(P\in\Phi_0)>c$ for all sufficiently large $M$.
\end{lemma}

\proof We first note that if $22\Delta^3\geq M_2$, then since $M$ is large enough and $\Delta^4=O(M)$, we have $M_2\slash M \to 0$. So we only need to consider the case when $B_1$ and $B_2$ are defined by~\eqn{bounds}.

If $\Delta^4=o(M)$ then the claim follows by~\cite{mckay90}*{Lemmas 2 and $3^\prime$}. If $\Delta^4=\Theta(M)$ then $P$ contains $O(\Delta^4/M)$ triple edges in expectation, whereas the expected number of multiple edges of higher multiplicity in the pairing is  bounded by $o(1)$.  Similarly, the expected number of loops of multiplicity at least 2 is $o(1)$.  In the case that the expected number of triple edges is asymptotically a positive constant, the standard method of moments can be used to show that the joint distribution of the numbers of triple edges, double edges and loops are asymptotically  independent Poisson variables. This implies our assertion. See also the discussion of this case in the proof of~\cite{mckay90}*{Theorem 3}. \qed

The first step of our algorithm is to use the configuration model to generate a uniformly random pairing $P\in\Phi(\bfd)$. Proceed if $P\in\Phi_0$. Otherwise, reject $P$ and restart the algorithm. This type of rejection is called {\em initial rejection}.  By Lemma~\ref{lem:initial}, this initial rejection stage takes only $O(1)$ rounds in expectation before successfully producing a multigraph $G=G(P)$ with at most $B_2$ double edges, at most $B_1$ loops, and no multiple loops or  edges of multiplicity higher than two. Then the algorithm calls two procedures, \NoLoops~and \NoDoubles. Each of these is composed of a sequence of {\em switching steps}. In each switching step,  a loop (in \NoLoops) or a double edge (in \NoDoubles) will be removed using the corresponding switching operation in the procedure.

\begin{algorithm}[H]
{\bf{Algorithm}}
{\Gen $(n, \bfd)$:}\\
Generate a uniformly random pairing $P\in \Phi(\bfd)$.\\
Reject $P$ if $P\notin \Phi_0$ (initial rejection) and otherwise set $G=G(P)$\;
  \NoLoops($G$)\;
\NoDoubles($G$).
\end{algorithm}
\ss
Various types of rejections may occur in procedures \NoLoops~and \NoDoubles. In all cases, if a rejection occurs then the algorithm restarts from the first step.

Let $\bfm=(m_1,m_2)$ and $\G_{\bfm}$ be the set of multigraphs with degree sequence $\bfd$, $m_1$ loops, $m_2$ double edges and no other types of multiple edges. The following lemma guarantees uniformity of the multigraph obtained after initial rejection.\begin{lemma}\label{lem:initialunif}
Let $P$ be a uniformly random pairing in $\Phi_0$. Let $\bfm=(m_1,m_2)$ where $m_1\le B_1$ and $m_2\le B_2$. Conditional on the number of loops and double edges in $P$ being $m_1$ and $m_2$, $G(P)$ is uniformly distributed over $\G_{\bfm}$.
\end{lemma}
\proof This follows from the simple observation that every pairing in $\Phi_0$ appears with the same probability, and every multigraph in $\G_{\bfm}$ corresponds to exactly $\prod_{i=1}^n d_i!/2^{m_1+m_2}$ distinct pairings.\qed
 
Note that if $ 22\Delta^3\geq M_2$, then $B_1=0$, $B_2=0$ and so \Gen\ never calls \NoLoops\ or \NoDoubles . By Lemma~\ref{lem:initialunif}, output of \Gen\ is a uniformly distributed in $\G_{0,0}$. Also, by Lemma~\ref{lem:initial}, \Gen\ restarts constant number of times in expectation before outputting a graph. Hence, in this case we proved Theorem~\ref{thm:main}. For the rest of this section we assume $ 22\Delta^3< M_2$.

In the next subsection we define the procedure \NoLoops. This procedure uses the same switchings as in~\cite{mckay90} (but applied to multigraphs rather than pairings) to reduce the number of loops to 0. 

\subsection{\NoLoops}

\begin{definition}[$\ell$-switching]
For a graph $G\in \G_{m_1,m_2}$, choose five distinct vertices $v_1,\ldots , v_5$ such that
\vspace{-\topsep}
\begin{itemize}
\setlength{\itemsep}{0em}
\item there is a loop on $v_2$.
\item $v_1v_4$ and $v_3 v_5$ are single edges;
\item there are no edges between $v_1$ and $v_2$, $v_2$ and $v_3$, $v_4$ and $v_5$.
\end{itemize}

An $\ell$-switching replaces loop on $v_2$ and edges $v_1v_4$, $v_3v_5$, by edges $v_1v_2$, $v_2v_3$ and $v_4v_5$. 
\end{definition}
See Figure~\ref{f:l-switching} for an illustration of an $\ell$-switching. Note that this switching is the same as the one used in~\cite{mckay90}, except performed on graphs, not pairings.

\begin{figure}[H]
	\begin{center}
	\includegraphics[trim={3cm 21cm 3cm 2.5cm},clip]{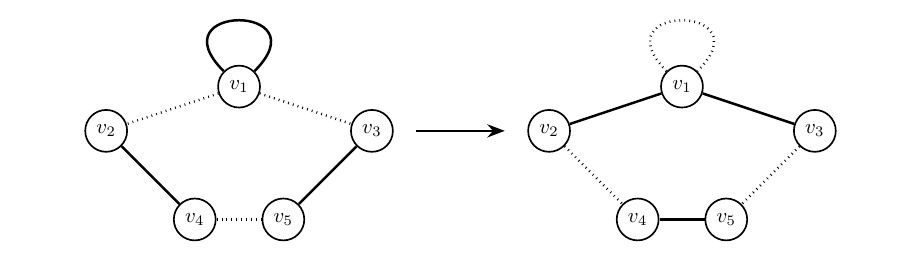}
	\caption{$\ell$-switching.}
\label{f:l-switching}
\end{center}	
\end{figure}

Let $f_\ell(G)$ be the number of $\ell$-switchings that can be performed on $G$. We will specify a parameter $\overline{f}_{\ell}(\bfm)$ such that 
\[
f_\ell(G)\le \overline{f}_{\ell}(\bfm) \quad \mbox{for all $G\in\G_{\bfm}$}.
\]
In each switching step, a uniformly random switching $S$ converting $G\in\G_{m_1,m_2}$ to some $G'\in\G_{m_1-1,m_2}$ is selected. An f-rejection occurs with probability $1-f_\ell(G)\slash \overline{f}_{\ell}(\bfm)$.  We will next describe how to use incremental relaxation to do b-rejection. If $S$ is neither f-rejected nor b-rejected, then $S$ will be performed in this switching step. 

  We first give some notation. In a multigraph, a {\em   (simple) ordered edge} is an ordered pair of vertices $(u,v)$ such that $uv$ is a (simple) edge in the multigraph. Similarly, a {\em  (simple) ordered $i$-path} is an ordered set of vertices $(u_1,\ldots, u_{i+1})$ such that $u_1u_2\cdots u_{i+1}$ forms a (simple) $i$-path in the multigraph. 
 
Define   $b_\ell(G',\emptyset)$ to be the number of  simple ordered $2$-paths $uvw$ in $G'$ such that there is no loop on $v$. For a simple ordered 2-path $uvw$ in $G'$ define $b_\ell(G', uvw)$ to be the number of simple ordered edges $u'w'$ in $G'$ that are vertex disjoint from $uvw$ and such that $uu'$ and $ww'$ are non-edges. For $\bfm=(m_1-1,m_2 )$ let $\LB_\ell(\bfm; 0)$ and $\LB_\ell(\bfm; 1)$ be lower bounds on $b_\ell(G',\emptyset)$ and $b_\ell(G',uvw)$ respectively over all $G'\in \G_{\bfm}$ and all simple ordered 2-paths $uvw$ in $G'$.  Positive constants $\LB_{\ell}(\bfm; 0)$ and $\LB_{\ell}(\bfm; 1)$ will be defined in Section~\ref{sec:parloops}. Any switching $S$ that can be used to create a fixed multigraph  $G'\in  \G_{m_1-1,m_2}$ from multigraphs in $\mathcal{G}_{m_1,m_2}$ can be identified with the  ordered set of vertices $\OV_2(S)=(v_1,\ldots,v_5)$ whose adjacencies were changed by $S$. Set $\OV_0(S)=\emptyset$ and $\OV_1(S)=(v_1,v_2,v_3)$.

Informally, each iteration of \NoLoops\ starts with a multigraph $G\in\G_{m_1,m_2}$ and chooses a random $\ell$-switching $S$ that converts $G$ to some $G'\in\G_{m_1-1,m_2}$. In terms of the notation defined in Section~\ref{s:deanchor}, each such switching $S$ can be viewed as an $H$-anchoring of $G'$, where $H$ is a graph on the right side of Figure~\ref{f:l-switching} (with positive signs on solid edges, and negative signs on dashed edges). \NoLoops\ then performs f-rejection, after which every pair $(G',\OV_2(S))$ (denoting an $H$-anchoring of $G'$), where $G'\in \G_{m_1-1,m_2}$ and $S$ is an $\ell$-switching that creates $G'$, arises with the same probability. After that \NoLoops\ sequentially relaxes constraints enforced by $H$-anchoring of $G'$ by performing a b-rejection.
The following is the formal description of \NoLoops.

\begin{algorithm}[H]
{\bf{procedure}}
{\NoLoops $(G)$:}\\
\While{$G$ \text{has a loop}}{
	let $\bfm=(m_1,m_2)$ be such that $G\in\mathcal{G}_{\bfm}$\;
	obtain $(G^\prime,\OV_2(S))$ from $G$ by performing a random $\ell$-switching $S$ on $G$\;
	f-rejection: {\textbf{restart}} with probability $1-\frac{f_\ell(G)}{\overline{f}_{\ell}(\bfm)}$\;
	 $\bfm\gets( m_1-1,m_2)$\; 
	 b-rejection: {\textbf{restart}} with probability $1-\frac{\LB_\ell(\bfm;0)\LB_\ell(\bfm;1)}{b_\ell(G,\OV_0(S))b_\ell(G,\OV_1(S))}$\; 
	$G \gets G^{\prime}$\;
    }
\end{algorithm}

 In Section~\ref{sec:uniformity} we show that if $G$ is distributed uniformly at random in $\G_{m_1,m_2}$, the output of \NoLoops (G) is uniform in $\G_{0,m_2}$. We do this by showing that the quantities $b_\ell(G,\OV_0(S))$ and $b_\ell(G,\OV_1(S))$ defined above coincide with the quantities $b(G,C_1)$ and $   b(G,C_1,C_2)$ in an application of  Corollary~\ref{cor:Deanchor}.

\subsubsection{Parameters in \NoLoops}\label{sec:parloops}
We now specify the values of the parameters mentioned above, which will be shown in the following lemma to satisfy the required inequalities.
Define
\begin{align*}
    \overline{f}_\ell(\bfm)&=m_1 M^2,\quad
    \LB_{\ell}(\bfm;1)&=M\left(1-\frac{6\Delta^2-4\Delta}{M}\right),\quad
    \LB_{\ell}(\bfm;0)&=M_2\left(1-\frac{8m_2\Delta+m_1\Delta^2}{M_2}\right).
\end{align*}
Recall that we assumed $22\Delta^3 < M_2$ and so $\LB_{\ell}(\bfm;0)$ and $\LB_{\ell}(\bfm;1)$ are positive constants. The following Lemma establishes necessary bounds on $b_\ell(G,\emptyset)$, $b_\ell(G,uvw)$ and $f_\ell(G)$.
\begin{lemma}\label{lemma:l-bounds}
Let $G\in \G_{m_1,m_2}$ with $m_1\leq M_2 \slash M$ and $m_2\leq M_2^2\slash M^2$. For any simple ordered 2-path $v_1v_2v_3$ in $G$, we have
\begin{align*}
\LB_\ell(\bfm;0)\leq  b_\ell(G,\emptyset) \leq M_2,\\
\LB_\ell(\bfm;1)\leq  b_\ell(G,v_1v_2v_3) \leq M.
\end{align*} For forward $\ell$-switchings  
$$m_1M^2(1-\frac{11\Delta^2-4\Delta+4}{M})\leq  f_\ell(G) \leq \overline{f}_\ell(\bfm).$$
\end{lemma}
The proof of Lemma~\ref{lemma:l-bounds} is postponed to Section~\ref{s:lemmasproof}.
This completes the description of \NoLoops.

\subsection{\NoDoubles}\label{subsec:nodoubles}
After \NoLoops~is finished, we have a multigraph $G\in \G_{0,m_2}$. 
Next we describe how to reduce the number of double edges in $G$.

\begin{definition}[d-switching]
For a graph $G\in \G_{0,m_2}$, choose six distinct vertices $v_1,\ldots , v_{6}$ such that
\vspace{-\topsep}
\begin{itemize}
\setlength{\itemsep}{0em}
\item there is a double edge between $v_2$ and $v_5$.
\item $v_1v_{4}$, $v_3v_6$, are single edges;
\item the following are non-edges: $v_1v_2$, $v_2v_3$, $v_4v_5$, $v_5v_6$.
\end{itemize}
A $d$-switching replaces double edges between $v_2v_5$ and edges $v_1v_{4}$, $v_3v_6$, by edges $v_1v_2$, $v_2v_3$, $v_4v_5$, $v_5v_6$.
\end{definition}
See Figure~\ref{f:d-switching} for an illustration.

\begin{figure}[H]
	\begin{center}
	\includegraphics[trim={3cm 21.5cm 3cm 2.5cm},clip]{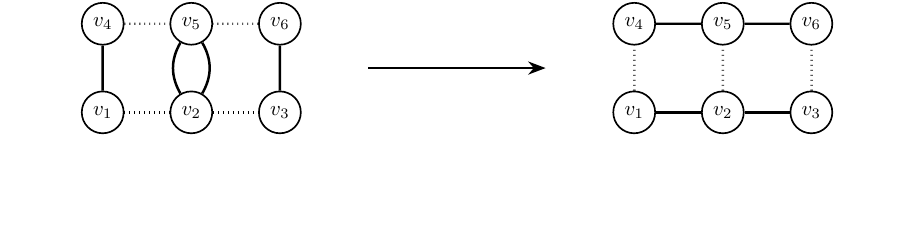}
	\caption{$d$-switching.}
\label{f:d-switching}
\end{center}	
\end{figure}

For a graph $G\in \G_{\bfm}$, we use notation $f_d(G)$ for the number of ways to perform a $d$-switching on $G$. We will specify $\overline{f}_{d}(\bfm)$ such that 
\[
f_d(G)\le \overline{f}_{d}(\bfm) \quad \mbox{for all $G\in\G_{\bfm}$}.
\]
In each switching step, a uniformly random switching $S$ converting $G\in\G_{0,m_2}$ to some $G'\in\G_{0,m_2-1}$ is selected. An f-rejection occurs with probability $1-f_d(G)\slash \overline{f}_{d}(\bfm)$. 
 
The incremental relaxation scheme for b-rejection is analogous to that in \NoLoops. Define   $b_d(G',\emptyset)$ to be the number of simple ordered $2$-paths $uvw$  in $G'$. For a simple ordered 2-path $uvw$ in $G'$ define $b_d(G', uvw)$ to be the number of simple ordered 2-paths $u'v'w'$ that are vertex disjoint from $uvw$ such that $uu'$, $vv'$ and $ww'$ are non-edges.

For $\bfm=(0,m_2-1)$ let $\LB_d(\textbf{m}; 0)$ and $\LB_d(\textbf{m}; 1)$ be positive lower bounds (to be specified in Section~\ref{sec:pardoubles}) on $b_d(G',\emptyset)$ and $b_d(G',uvw)$ over all $G'\in \G_{\bfm}$ and simple ordered 2-paths $uvw$ in $G'$. For a $d$-switching $S$ let $\OV_2(S)=(v_1,\ldots,v_6)$ be the vertices whose adjacencies were changed by $S$. Set $\OV_0(S)=\emptyset$ and $\OV_1(S)=(v_1,v_2,v_3)$. 

\begin{algorithm}[H]
{\bf{procedure}}
{\NoDoubles $(G)$:}\\
\While{$G$ \text{has a double edge}}{
	let $\bfm=(0,m_2)$ be such that $G\in\mathcal{G}_{\bfm}$\;
	obtain $(G^\prime,\OV_2(S))$ from $G$ by performing a random $d$-switching $S$ on $G$\;
	f-rejection: {\textbf{restart}} with probability $1-\frac{f_d(G)}{\overline{f}_{d}(\bfm)}$\;
	$\bfm\gets(0,m_2-1)$\;
	b-rejection:  {\textbf{restart}} with probability $1-\frac{\LB_d(\bfm;0)\LB_d(\bfm;1)}{b_d(G,\OV_0(S))b_d(G,\OV_1(S))}$\;
	$G \gets G^{\prime}$\;
     }
\end{algorithm}
 As in case of \NoLoops\ , In Section~\ref{sec:uniformity} we show the desired uniformity property holds for \NoDoubles\ .

\subsubsection{Parameters for \NoDoubles}\label{sec:pardoubles}
Define
\begin{align*}
\overline{f}_{d}(\bfm)&=2m_2M^2,\quad
    \LB_{d}(\bfm;0)&=M_2\left(1-\frac{8m_2\Delta}{M_2}\right),\quad
    \LB_{d}(\bfm;1)&=M_2\left(1-\frac{4m_2(2\Delta-3)+3\Delta^3}{M_2}\right).
\end{align*}
Note that $\LB_d(\bfm;0)$ and $\LB_d(\bfm;0)$ are positive constants, as in Section~\ref{sec:parloops}.

\begin{lemma}\label{lemma:d-bounds}
Let $G\in \G_{0,m_2}$. Then for any simple ordered 2-path $v_1v_2v_3$ in $G$ we have
\begin{align*}
 & \LB_d(\bfm;0)\leq  b_d(G,\emptyset) \leq M_2,\\  
  &\LB_d(\bfm;1)\leq  b_d(G,v_1v_2v_3) \leq M_2,\\  
  & 2m_2M^2\left(1-\frac{12\Delta^2-4\Delta+8}{M}\right)\leq  f_d(G) \leq \overline{f}_d(\bfm).
\end{align*}
\end{lemma}
 The proof of Lemma~\ref{lemma:d-bounds} is postponed to Section~\ref{s:lemmasproof}.

\subsection{Uniformity}\label{sec:uniformity}

\begin{thm}\label{thm:uniformity}
\Gen~generates graphs with degree sequence $\bfd$ uniformly at random.
\end{thm}
\begin{proof}
 We start the proof by showing that b-rejection in both \NoLoops\ and \NoDoubles\ can be performed as \Deanchor\ for appropriate choice of $\F,S_1,S_2$. We deal here with \NoDoubles\ only, as the issues with \NoLoops\ are identical.

Let $\S$ be the set of $d$-switchings that convert a multigraph in $\G_{0,m_2}$ to some multigraph in $\mathcal{G}_{0,m_2-1}$. 
Recall that switching $S\in \S$ can be identified with an ordered set of vertices $\OV_2(S)=(v_1,\ldots,v_6)$ whose adjacencies were changed by $S$, and $\OV_0(S)=\emptyset$, $\OV_1(S)=(v_1,v_2,v_3)$. 

Let $\F=\G_{0,m_2-1}$ and let $v_1, \ldots, v_6$ be distinct vertices. Using   the notation $\{ \}^*$ to denote a multiset, and $E_1(G)$ to denote the set of simple edges in $G$, define
\bean
C_1^{(v_1,v_2,v_3)} &=&\{\wt{G}\in \F :  v_1v_2, v_2v_3\in E_1(\wt{G})\},\\
C_2^{(v_1,\ldots,v_6)}& = &\{\wt{G}\in C_1^{(v_1,v_2,v_3)} :v_4v_5, v_5v_6\in E_1(\wt{G}) \mbox{ and }   v_1v_4, v_2v_5, v_3v_6\notin E (\wt{G})\},\\
S_1&=&\{C_1^{(v_1,v_2,v_3)}: v_1,v_2,v_3 \mbox{ all distinct}\}^*,\\
S_2&=&\{C_2^{(v_1,\ldots,v_6)}: v_1,\ldots,v_6 \mbox{ all distinct}\}^*,\\
 \F_2&=&\{(G, C_1^{(v_1,v_2,v_3)}, C_2^{(v_1,\ldots,v_6)} ):  v_1,\ldots,v_6 \mbox{ all distinct}, G\in C_2^{(v_1,\ldots,v_6)} \},\\
 \F_0&=&\F.
\eean

\remove{
For a given 6-tuple of vertices $(v_1,\ldots,v_6)$ let $C_1^{(v_1,v_2,v_3)}$ be a subset of $\wt{G}\in C_1^{(v_1,v_2,v_3)}$ such that $v_4v_5v_6$ form a simple path in $\wt{G}$ and edges $v_1v_4$, $v_2v_5$, $v_3v_6$ are not present in $\wt{G}$.  Set $S_1$ to be a set consisting of $C_1^{(v_1v_2v_3)}$ for all possible triplets of different vertices $(v_1,v_2,v_3)$. Note that $C_1^{(v_1v_2v_3)}=C_1^{(v_3v_2v_1)}$ and so $S_1$ is a multiset with each element being of multiplicity two. Finally, set $S_2$ to be a set consisting of $C_2(v_1, \ldots,v_6)$ for all possible 6-tuples of different vertices ${v_1,\ldots,v_6}$. Similarly, $S_2$ is a multiset with each element being of multiplicity four. Define $\F_2 \subseteq \F \times S_1 \times S_2$ be such that for all $(G,C_1,C_2)\in \F_2$ there is a 6-tuple $v_1,\ldots,v_6$ such that  $(G,C_1,C_2)=(G, C_1(v_1,\ldots,v_3), C_2(v_1,\ldots,v_6) )$.}

Recall that $$\F_1 = \{(G, C_1^{(v_1,v_2,v_3)}): (G, C_1^{(v_1,v_2,v_3)},C_2^{(v_1,\ldots,v_6)})\in \F_2 \mbox{ for some } v_4,v_5,v_6\}$$
We now show that $$\F_1 = \{(G, C_1^{(v_1,v_2,v_3)}):  v_1,v_2,v_3 \mbox{ all distinct}, G \in C_1^{(v_1,v_2,v_3)}\}.$$ 
Indeed, for a given simple ordered 2-path $v_1v_2v_3$ in $G$, the number of simple ordered 2-paths $v_4v_5v_6$ such that $v_1v_4$, $v_2v_5$ and $v_3v_6$ are non-edges is equal to $b_d(G, v_1v_2v_3)$ and is at least one according to Lemma~\ref{lemma:d-bounds}. So for every pair $(G,C_1^{(v_1,v_2,v_3)})$ with $G\in C_1^{(v_1,v_2,v_3)}$ there exists a simple ordered 2-path $v_4v_5v_6$, such that $(G,C_1^{(v_1,v_2,v_3)}, C_1^{(v_1,\ldots,v_6)})\in \F_2$, which establishes the desired claim for $\F_1$.

Similarly we have
$$\F_0 = \{G : (G, C_1^{(v_1,v_2,v_3)})\in \F_1 \mbox{ for some } v_1,v_2,v_3\}.$$

If $S$ is a switching from $G$ to $G'$, then $G'\in C_1^{\OV_1(S)}$ and $G'\in C_2^{\OV_2(S)}$ so $(G', C_1^{\OV_1(S)},C_2^{\OV_2(S)} )$ belongs to $\F_2$. So every pair $(G',\OV_2(S))$, where switching $S\in \S$ creates $G'$, can be identified with an element $(G', C_1^{\OV_1(S)},C_2^{\OV_2(S)} )\in \F_2$, hence we can apply \Deanchor\ to $(G',\OV_2(S))$. In this setup, the quantities $b(G')$ and $b(G', C_1^{\OV_1(S)})$ (as in Section~\ref{s:deanchor}) are equal to $b_d(G',\OV_0(S))$ and $b_d(G',\OV_1(S))$ respectively. (Recall the definitions for $b_d(G',\OV_0(S))$ and $b_d(G',\OV_1(S))$ in Section~\ref{subsec:nodoubles}.) It remains to note that we can set $\LB(i)=\LB_d(\bfm;i)$ for $i\in \{0,1\}$ where $ \bfm=(0,m_2-1)$.

According to Corollary~\ref{cor:Deanchor}, \Deanchor$(G',C_1^{\OV_1(S)}, C_2^{\OV_2(S)} )$ outputs $G'$ with probability 
$$\LB(0)\LB(1)\slash  b(G')b\big(G', C_1^{\OV_1(S)}\big),$$
which is exactly equal to the probability that $G'$ is not b-rejected in \NoDoubles.

Hence b-rejection in \NoDoubles\ is just an effective implementation of~\Deanchor$(G',C_1^{\OV_1}, C_2^{\OV_2})$. In view of  Corollary~\ref{cor:de-anchor} we have the following.
 \begin{cor}\label{cor:uniform}
Let $m_1, m_2$ be non-negative integers and let $(G',\OV_2(S))$ be chosen u.a.r from the class of all pairs $(\wt{G},\OV_2(\wt{S}))$, where $\wt{G}\in \G_{m_1,m_2}$  and $\wt{S}$ is an $\ell$-switching (or $d$-switching, if  $m_1=0$) that creates $\wt{G}$. If $(G',\OV_2(S))$ is not b-rejected by \NoLoops\ (or \NoDoubles, respectively), then $G'$ is uniform in $\G_{m_1,m_2}$.
\end{cor} 

Now we are ready to prove the theorem. Assume that we initially generated a graph $G_0\in \G_{m_1,m_2}$ for some $m_1\leq M_2 \slash M$ and $m_2\leq M_2^2\slash M^2$. 

We say that a graph $G$ was reached in \NoLoops\ if a switching creating $G$ was selected in a switching step, and $G$ was not rejected. Let $G_t$ denote the multigraph reached after $t$ switching steps of \NoLoops, if no rejection occurred (let $G_t=\emptyset$ if a rejection occurs during the $t$-th step or earlier). We will prove by induction on $t$, that conditional on $G_t\in\G_{m_1-t,m_2}$, $G_t$ is uniformly distributed in $G_{ m_1-t,m_2}$. 

The base case $t=0$ holds by Lemma~\ref{lem:initialunif}. Assume $t\ge 0$ and $G_t$ is uniformly distributed in $\G_{ m_1-t ,m_2}$. Then, there exists $\sigma_{ m_1-t,m_2}$ such that the probability that $G_t=G$ is equal to $\sigma_{ m_1-t ,m_2}$, for every $G\in\G_{ m_1-t,m_2}$. 
Now, for every $G' \in \G_{ m_1-t-1 ,m_2}$ and every $\ell$-switching $S$ that results in $G'$, the probability that $(G',\OV_2(S))$ was obtained during the $(t+1)$-st iteration of \NoLoops\ and not f-rejected is equal to 
$$\sigma_{ m_1-t ,m_2}\frac{1}{f_\ell(G)}\frac{f_\ell(G)}{\UB_{\ell}( m_1-t ,m_2)}=\frac{\sigma_{  m_1-t ,m_2}}{\UB_{\ell}( m_1-t ,m_2)}.$$
So, $(G_{t+1},\OV_2(S))$ is uniform in class of all pairs $(\wt{G},\OV_2(\wt{S}))$, where $\wt{G}\in \G_{ m_1-t-1 ,m_2}$ and $\wt{S}$ is an $\ell$-switching that creates $\wt{G}$. By Corollary~\ref{cor:uniform}, if $(G_{t+1},\OV_2(S))$ is not b-rejected then $G_{t+1}$ is uniform in $\G_{ m_1-t-1 ,m_2}$. Inductively, the output of \NoLoops\ is uniform in $\G_{0,m_2}$ provided no rejection. This holds as well for \NoDoubles. Therefore, \Gen~ generates every graph in $\G_{0,0}$ with the same probability.\qed

\end{proof}


\subsection{Time and space complexity}
\label{sec:time}

\begin{lemma}\label{lem:fb-rej}
The probability of an f- or b-rejection during a single run of \Gen~is at most $1-\exp(-O(\Delta^4/M))$.
\end{lemma}
\begin{proof} 
First, note that if $M_2<M,$ or $22\Delta^3\geq M_2$ then both $B_1, B_2$ are smaller than 1 and \NoLoops\ and \NoDoubles\ are never called, since in these cases after initial rejection we obtain a uniformly random simple graph. Assume  $M_2\geq M$. We first deal with \NoLoops. 

By Lemma~\ref{lemma:l-bounds}, the probability of no rejection in a single switching step of \NoLoops~is at least
\begin{align*}
\frac{f_\ell(G)}{\UB_{\ell}(\bfm)}\frac{\LB_\ell(\bfm;0) \LB_\ell(\bfm;1)}{b_{\ell}(G',\OV_0(S))b_{\ell}(G',\OV_1(S))}&\geq \left(1-O\left(\frac{\Delta^3}{M_2}\right)\right)\left(1-O\left(\frac{\Delta^2}{M}\right)\right)^2=1-O\left(\frac{\Delta^3}{M_2}\right).
\end{align*}
Since there are at most $m_1\leq M_2 \slash M$ iterations of \NoLoops, the probability of no rejection during \NoLoops~is at least
$$\left(1-O\left(\frac{\Delta^3}{M_2}\right)\right)^{M_2/M}=\exp\left(-O\left(\frac{\Delta^3}{M_2}\right)\frac{M_2}{M}\right)=\exp(-O(\Delta^3/M)).$$
Similarly, for \NoDoubles, the probability that no rejection occurs in a single switching step, assuming no rejection occurring before, is at least
\begin{align*}\frac{f_d(G)}{\UB_{d}(\bfm)}\frac{\LB_d(\bfm;0) \LB_d(\bfm;1)}{b_{d}(G',\OV_0(S))b_{d}(G',\OV_1(S))}&\geq \left(1-O\left(\frac{\Delta^2}{M}\right)\right)\left(1-O\left(\frac{\Delta^3}{M_2}\right)\right)=1-O\left(\frac{\Delta^3}{M_2}\right).
\end{align*}
As there are at most $m_2\leq M^2_2 \slash M^2$ iterations of \NoDoubles, the probability of no rejection during \NoDoubles~ is at least
$$\left(1-O\left(\frac{\Delta^3}{M_2}\right)\right)^{M_2^2/M^2}=\exp\left(-O\left(\frac{\Delta^3}{M_2}\right)\frac{M^2_2}{M^2}\right)=\exp(-O(\Delta^4/M)).$$
Hence, the probability of any rejection during a single run of \NoLoops, or \NoDoubles\ is $1-\exp(-O(\Delta^3/M)-O(\Delta^4/M))=1-\exp(-O(\Delta^4/M))$.\qed
\end{proof}
\ss

Now we complete the proof for Theorem~\ref{thm:main}, which follows from Theorem~\ref{thm:uniformity} and the following.
\begin{thm}\label{thm:complexity}
Provided $\Delta^4=O(M)$, the expected run time of \Gen~is $O(M)$. Space complexity of \Gen\ is $O(n^2)$.
\end{thm}
\begin{proof} 
We start with estimating space complexity. By implementing appropriate data structures (uninitialised adjacency matrix and sorted arrays) we may assume that it takes constant time for checking adjacency of the vertices and to access the list of neighbours. We also store the positions of multiple loops and double edges. In total our space complexity is bounded by $O(n^2+\Delta+\Delta^2)=O(n^2)$.

By Lemmas~\ref{lem:initial} and~\ref{lem:fb-rej}, \Gen\ restarts a constant number of times in expectation before outputting a graph.
So we only need to estimate the run time for a single run of \Gen. The initial generation of $P$ takes $O(M)$ time. The positions of all loops and multiple  edges can be stored along with the generation of $P$, so the detection of triple edges and double loops requires negligible time comparatively. Assuming the initial pairing survives initial rejection, the numbers of loops and double edges can be updated in constant time after each switching.
We need to show that both \NoLoops~and \NoDoubles~can be implemented in time $O(M)$.

We first deal with the implementation of the f-rejection step. Instead of computing $f_\ell(G)$, we choose a random loop (on a vertex $v_1$) and then independently  choose two uniformly random ordered edges $v_2v_4$ and $v_3v_5$ (this all can be done in time $O(1)$). If on the corresponding $\OV_2=(v_1,\ldots,v_5)$ we cannot perform an $\ell$-switching due to some vertices colliding, forbidden edges being present, or single edges being actually loops or double edges, then we reject such $\OV_2$ (checking if a switching can be performed on $\OV_2$ can be done in constant time). There are $\overline{f}_{\ell}(\bfm)=m_1M^2$ ways to choose such a set $\OV_2$, and the probability of accepting $\OV_2$ is exactly $f_\ell(G)\slash \overline{f}_{\ell}(\bfm)$. 

Similarly, for f-rejection in \NoDoubles, we choose a random ordered double edge $v_2v_5$, and independently choose two uniformly random ordered edges (repetition allowed) to be $v_1v_4$ and $v_3v_{6}$ and reject the corresponding set $\OV_2=(v_1,\ldots,v_{6})$ if a $d$-switching cannot be performed on it. There are exactly $ \overline{f}_{d}(\bfm)=2m_2M^2$ total choices for $\OV_2$ and probability of accepting it is exactly $f_d(G)\slash \overline{f}_{d}(\bfm)$.

Implementation of the b-rejection step is more complicated;  this requires computing $b_{\ell}(G,\OV_i(S))$ and $b_{d}(G,\OV_i(S))$. We start with computing $P_{2}(G)$, which we define to be the number of simple ordered $2$-paths $uvw$ in $G$ with no loop on $v$. We can do this initially in time $O(M)$ by going through all vertices $v_i$ which have no loop on them and checking how many single edges are incident to $v_i$. (We are counting paths from their middle vertex: if there are $k$ such edges, $v_i$ contributes $k(k-1)$ to the count of paths.) After each $\ell$-switching and $d$-switching, $P_{2}(G)$ can be updated in time $O(\Delta)$. Indeed, each switching affects the adjacency of at most six pairs of vertices. For each adjacency change we can count the 2-paths it affects in time $O(\Delta)$. $P_{2}(G)$ has to be updated at most $m_1+m_2=O(\Delta^2)$ times, so the initial calculations and the update of $P_{2}(G)$ can be done in time $O(M)$ in total.

Now we prove that $b_\ell(G',\OV_1(S))$ can be calculated in time $O(\Delta^2)$. Indeed, for $\OV_1=(v_1,v_2,v_3)$, $b_\ell(G',\OV_1(S))$ is the number of simple ordered edges $e=(uv)$ so that $e\cap \OV_1=\emptyset$ and there is no edge $v_1u$ or $v_3v$. Thus $b_\ell(G',\OV_1(S))$ can be estimated as $M$ minus the number of ``bad" choices of $e$, ie\ choices that violate at least one of the three conditions. This number of bad choices can be calculated in time $O(\Delta^2)$ by going through the 2-neigborhood of $v_1$ and $v_3$.
On the other hand,  $b_{\ell}(G',\OV_0(S))$ is already given by $P_{2}(G')$, and thus does not require additional computation.
 
For b-rejections in \NoDoubles, we need to compute $b_d(G',\OV_0(S))$ and $b_d(G',\OV_1(S))$. Again, the value of $b_{d}(G',\OV_0(S))$ is already given by $P_{2}(G')$. 
We claim that $b_d(G',\OV_1(S))$ can be calculated in time $O(\Delta^2)$. Assume $\OV_1(S)=(v_1,v_2,v_3)$ is given and is fixed and we are choosing $(v_4,v_5,v_6)$. The number of simple ordered paths $(v_4,v_5,v_6)$ is given by $P_{2}(G')$, so we need to subtract from this the number of paths where some vertices collide with $\OV_1(S)$, or there is an edge (or double edge) between $v_2v_5$, or $v_1v_4$, or $v_3v_6$. Formally, let $B_{i,j}$ with $i\in\{1,2,3\}$ and $j\in\{4,5,6\}$ be the set of simple ordered $2$-paths $v_4v_6v_5$ such that $v_i$ coincides with $v_j$, and let $E_1$, $E_2$, $E_3$ be the sets of simple ordered $2$-paths $v_4v_5v_6$ such that $v_1v_4$, $v_2v_5$, or $v_3v_6$ is an edge (or a double edge), respectively. Then
$$b_d(G',\OV_1(S))=P_2(G')-|\cup_k E_k\setminus \cup_{i,j}B_{i,j}|-|\cup_{i,j}B_{i,j}|.$$
To estimate the size of $B=\cup_{i,j}B_{i,j}$ we can use the inclusion-exclusion formula. It is easy to see that no more than three different $B_{i,j}$ can have non-empty intersection, and each of the terms involving at least one of the $B_{i,j}$ can be computed in time $O(\Delta^2)$. Similarly, to estimate $\cup_k E_k\setminus B$ we use the formula $$|\cup_k E_k\setminus B|=|E_1\setminus B|+|E_3\setminus B|-|E_1\cap E_3\setminus B|+|E_2\setminus B \cup E_1\cup E_3|.$$ We only show in detail how to calculate the size of $E=(E_1\cap E_3) \setminus B$ in time $O(\Delta^2)$, as the sizes of the other three sets can be computed similarly. We run through all possible choices of $v_4v_5$ and show that, for each one, it takes constant time to count the vertices $v_6$ such that $v_4v_5v_6\in E$.

To start with, for each vertex $v$ of $G'$ let $f(v,v_3)$ denote the number of 2-paths  $v_3xv$    such that vertex $x$ is different from $v_1$ and $v_2$,  $xv_3$ is a single or double edge, and $vx$ is single edge. The values of $f(v,v_3)$ can be pre-computed for all $v$ in time $O(\Delta^2)$ by going through $x$ in the neighborhood of $v_3$ and   all $v$ in  the neighborhood of $x$. After that, to evaluate $|E|$, we go through the choices of $v_4$ as a neighbor of $v_1$ (at most $\Delta$ such choices), and $v_5$ as a neighbor of $v_4$ (again at most $\Delta$). For each choice of $v_4,v_5$ we first check if it is a valid choice for $E$, that is if $v_1v_4$ is an edge (or double edge) and if $v_4v_5$ is a single edge. (This can be done in constant time.) If given  $v_4,v_5$ is a valid choice for $E$, then there are exactly $f(v_5,v_3)$ choices for $v_6$ so that $v_4v_5v_6\in E$, if $v_4v_3$ is a non-edge, and there are exactly $f(v_5,v_3)-1$ choices for $v_6$ so that $v_4v_5v_6\in E$, if $v_4v_3$ is an edge (double or single).  Since going through all possible choices of $v_4v_5$ takes $O(\Delta^2)$ time, and  moreover given $v_4v_5$ it takes constant time to count the elements of $E$ of the form $v_4v_5v_6$, the size of $E$ can be calculated in time $O(\Delta^2)$.

Since \NoDoubles~runs for at most $\Delta^2$ iterations, and each iteration can be performed in time $O(\Delta^2)$, it takes at most $O(\Delta^4)=O(M)$ time for single run of $\NoDoubles$.

In conclusion, the expected run time of \Gen~is $O(M)$. \hfill \qed
\end{proof}

Alternatively, \Gen\ can be implemented (by using sorted adjacency listings for each vertex instead of adjacency matrix) so that the expected runtime is $O(M\log \Delta)$ and space complexity is $O(M)$.


\subsection{Proofs of Lemmas~\ref{lemma:l-bounds} and~\ref{lemma:d-bounds}}\label{s:lemmasproof}

The following lemma is used to estimate $b_\ell(G',\emptyset)$ and $b_d(G',\emptyset)$. 
\begin{lemma}\label{lemma:ineq2}
Let $G'\in \G_{m_1,m_2}$ be a graph with $m_1\leq M_2 \slash M$ and $m_2\leq M_2^2\slash M^2$. Then 
$$M_2\left(1-\frac{8m_2\Delta+m_1\Delta^2}{M_2}\right)\leq b_{\ell}(G', \emptyset)\leq M_2.$$
For $G'\in \G_{0,m_2}$ with $m_2\leq M_2^2\slash M^2$ we have
$$M_2\left(1-\frac{8m_2\Delta}{M_2}\right)\leq b_{d}(G', \emptyset)\leq M_2.$$
\end{lemma}

\begin{proof} Recall, that $b_{\ell}(G', \emptyset)$ is equal to the number of choices of simple ordered 2-path $uvw$ such that there is no loop on $v$. The same is true for $b_{d}(G', \emptyset)$. We first deal with $b_{\ell}(G', \emptyset)$.  In order to count the valid $2$-paths we first choose the vertex $v$ and then two distinct edges $uv$ and $vw$. There are at most $M_2$ ways to choose two adjacent edges in $G'$, and hence the upper bound. For the lower bound we have to subtract the choices for which either there is a loop on vertex $v$ (at most $m_1\Delta(\Delta-1)$ choices), or one of the  edges $uv$ and $vw$ is a double edge (at most $4m_2(2\Delta-3)$ choices, noting that for every double edge we may choose an edge from it in 2 ways and order it in 2 ways).    Hence the number of choices of $(u,v,w)$ that contribute to $b_\ell(G',\emptyset)$ is at least $$M_2-m_1\Delta(\Delta-1)-4m_2(2\Delta-3),$$ from which the lower bound follows.

The bounds for $b_{d}(G',\emptyset)$ follow by just setting $m_1=0$.\hfill \qed
\end{proof}
For the rest of this subsection set $\OV_1 = (v_1, v_2, v_3)$, where $v_1v_2v_3$ is a simple path in a multigraph with no loop on $v_2$.

\noindent {\bf Proof of Lemma~\ref{lemma:l-bounds}.\ }
First we deal with $b_\ell(G,\emptyset)$. Lemma~\ref{lemma:ineq2} implies that 
\begin{align*}
    \LB_\ell(\bfm; 0)\leq b_{\ell}(G, \emptyset) \leq M_2.
\end{align*}
The inequalities required for  $b_{\ell}(G,\OV_1)$ are 
\begin{align*}
    \LB_\ell(\bfm; 1)\leq  b_{\ell}(G,\OV_1) \leq M.
\end{align*}
There are at most $M$ choices for an ordered edge $e=(u,v)$ without any restrictions, hence the upper bound of $M$. Next, the choices of $e$ that do not contribute to $b_{\ell}(G,\OV_1)$ consist of one of the following three cases: 
\begin{itemize}[noitemsep,topsep=1pt]
\setlength\itemsep{0em}
\item[(i)] $e$ is a double edge or a loop;
\item[(ii)] $u\in \OV_1$ or $v\in \OV_1$ and not (i);
\item[(iii)] at least one of $uv_1$ and $vv_3$ is an edge and not (i), nor (ii).
\end{itemize}
There are at most $4m_2+2m_1$ edges $e$ that satisfy (i) (noting that loops count twice because the bound $M$ counts each edge once for each way to orient it); at most $6\Delta-4$ choices that satisfy (ii); at most $2(\Delta-1)^2$ choices that satisfy (iii). Hence the number of choices of $e$ that contribute to $b_{\ell}(G,\OV_1)$ is \ at least 
$$M-(4m_2+2m_1)-(2\Delta^2+2\Delta-2),$$ from which the lower bound follows (noting that the hypotheses of the lemma imply $m_1\le \Delta-1$ and $m_2\le (\Delta-1)^2$).

Turning to the estimation of $f_{\ell}(G)$, we first choose a vertex $v_2$ with a loop (in $m_1$ ways), and then ordered edges $v_1v_4$ and $v_3v_5$  (in at most $M$ ways each). Therefore, $f_{\ell}(G)\leq m_1M^2$. For the lower bound we need to subtract the following three cases:
at least one of $v_1v_4$ or $v_3v_5$ is a loop or a double edge (at most $m_1(4m_1+8m_2)M$); some of the vertices $v_1,\ldots,v_5$ coincide (at most $8m_1M\Delta$ such choices); or some of the edges $v_1v_2$, $v_2v_3$, and $v_4v_5$ are present (at most $3m_1\Delta^2M$ choices). Hence, there are at least $$m_1M(M-(4m_1+8m_2)-8\Delta-3\Delta^2)$$ $\ell$-switchings that can be applied to $G$. Again using $m_1\leq \Delta-1$ and $m_2\leq (\Delta-1)^2$, we obtain a lower bound for $f_{\ell}(G)$. \hfill \qed

\noindent {\bf Proof of Lemma~\ref{lemma:d-bounds}.\ }
Again, we deal with $b_d(G,\emptyset)$ and $b_d(G,\OV_1)$ first. Lemma~\ref{lemma:ineq2} implies that 
\begin{align*}
    \LB_d(\bfm; 0)\leq b_{d}(G, \emptyset) \leq M_2.
\end{align*}
We need to show that 
\begin{align*}
    \LB_d(\bfm; 1)\leq b_{d}(G, \OV_1) \leq M_2.
\end{align*}
Here $b_d(G',\OV_1)$ is the number of simple ordered 2-paths $uvw$ that do not intersect with $V_1$ and $v_1v$, $v_2u$, $v_3w$ are not edges. To choose $uvw$ we first choose the vertex $v$ and then different edges $uv$ and $vw$. There are at most $M_2$ ways to choose two adjacent edges in $G'$, hence the upper bound. For the lower bound, we have to subtract choices where any of the following holds: \begin{itemize}[noitemsep,topsep=1pt]
\setlength\itemsep{0em}
\item[(i)] at least one of $uv$ and $vw$ is a double edge;
\item[(ii)] some of the vertices $u,v,w$ coincide with some of vertices in $V_1$ and not (i);
\item[(iii)] at least one of edges $uv_1$,$vv_2$ and $wv_3$ is present and not (i), nor (ii).
\end{itemize}
There are at most $4m_2(2\Delta-3)$ choices for (i), at most $9\Delta^2-17\Delta+8$ choices for (ii), and at most $3\Delta^3-11\Delta^2+14\Delta-6$ choices for (iii). Hence the number of choices of $uvw$ that contribute to $b_\ell(G',\OV_1)$ is at least $$M_2-4m_2(2\Delta-3)-3\Delta^3+2\Delta^2+3\Delta-2,$$ from which the lower bound follows.

To estimate $f_{d}(G)$ we first choose an ordered double edge $v_2v_5$, then consecutively ordered edges $v_1v_4$ and  $v_3v_6$ so that $(v_1,\ldots, v_{6})=\OV_2(S)$ for some switching $S$. There are $2m_2$ ways to choose $v_2v_5$ and at most $M$ ways to choose each of the single edges. Therefore $f_{d}(G)\leq 2m_2M^2$. For the lower bound we need to subtract the choices in each of the following three cases:
at least one of $v_1v_4$ and $v_{3}v_{6}$ forms a double edge (there are at most $2m_2(8m_2)M$ such choices); some of the vertices $v_1,\ldots,v_{6}$ coincide (at most $2m_2(12M\Delta)$ choices); or some of the edges  $v_1v_2$, $v_2v_3$, $v_4v_5$ and $v_5v_6$ are present (at most $2m_2(4\Delta^2M)$ choices). Hence, there are at least
$$2m_2M(M-8m_2-12\Delta-4\Delta^2)$$ $d$-switchings that can be applied to $G$.
The  lower bound for $f_{d}(G)$ follows, using $m_2\leq (\Delta-1)^2$. \hfill \qed 


\section{Regular degree sequences}
\label{sec:regular}

In this section we aim at uniform generation of $d$-regular graphs where $d=o(\sqrt{n})$. In~\cite{gao17}, a uniform sampler called {\tt REG} was given which runs in time $O(d^3n)$ in expectation. Similar to \Gen, {\tt REG} first generates a uniformly random pairing which does not contain too many loops, double edges, or triple edges, and does not contain any multiple loops, or any multiple edges of multiplicity greater than three.\remove{The initial pairing generation takes time $O(dn)$ in expectation. Additionally to generating a pairing we calculate the number and locate the positions of loops, double edges and triple edges, as well as number of simple edges that are adjacent to each vertex, all in time $O(dn)$.} Then {\tt REG} goes through three ``phases'', reducing the loops, triple edges, and finally all double edges. Our new algorithm \reg\ has exactly the same structure,  employs the same switchings,  but has a  more efficient rejection scheme. The switchings in {\tt REG} are defined on pairings instead of on multigraphs. These two definitions are equivalent and effect parameters such as $f(G)$ and $b(G)$ by a constant factor in the two definitions. We refer the reader  to~\cite{gao17} for the description of {\tt REG}, which we do not reproduce here due to its length and complicated structure. For consistency, we will also define switchings on pairings in this section.

Thus, we will choose points in the vertices (instead of choosing vertices) and switch pairs involving these points. Instead of giving a formal definition we will only present a figure of the switchings, as the figures are self-explanatory. The choices of points are always made so that only the designated multiple edge, or loop is removed, without deleting any other multiple edges. Certain adjacency requirements are enforced so that the switching does not cause the creation of other multiple edges, unless specified. This is the same as for $\ell$-switchings and $d$-switchings in Section~\ref{sec:algorithm}.

The {\em first phase} reduces the number of loops. Our new algorithm \reg\ simply replaces that phase in~\cite{gao17} by procedure \NoLoops (adapted to pairings).

The {\em second phase} reduces the number of triple edges. The switchings used in~\cite{gao17} are in Figure~\ref{f:tripleOld}, and in \reg\ we use the same switchings, which we call $t$-switchings.

\begin{figure}[!tbp]
  \centering
 \hbox{\centerline{\includegraphics[width=7cm]{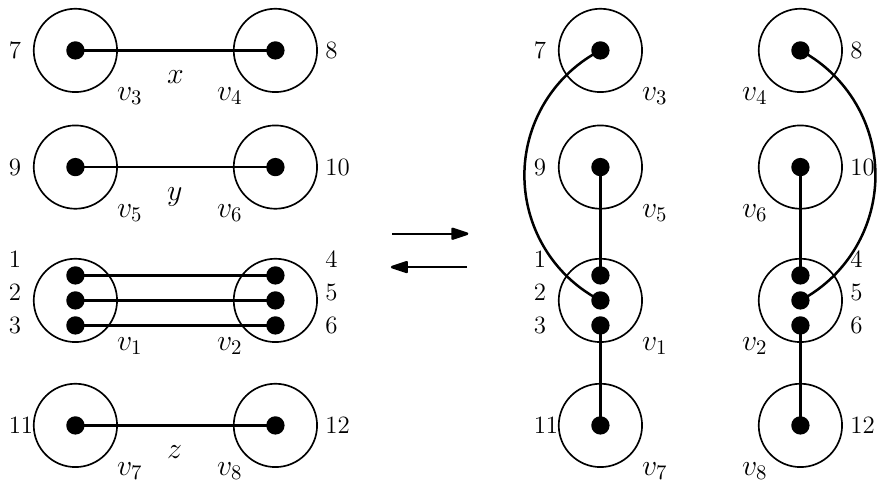}}}
\caption{t-switching used in~\cite{gao17}}
\label{f:tripleOld}
\end{figure}

Let $B_D$ and $B_T$ be the maximum numbers of double and triple edges permitted after the initial rejection. They were set in~\cite{gao17}*{eq.~(36)}. We keep this same definition for \reg. In particular, $B_D=O(d^2)$ and $B_T=O(\log n+d^3/n)$. 

Given a pairing $P$ that contains exactly $j$ triple edges, 
let $f_t(P)$ be the number of possible ways to perform a  $t$-switching on $P$. 

As before, for a switching $S\sim (P',\OV(S))$ from $P$ to $P'$, where $\OV(S)$ designates the set of points involved in the switching, define ordered subsets {\em of points} $\OV_0(S)=\emptyset$, 
$\OV_1(S)=(1,2,3,7,9,11)$ and $\OV_2(S)=\OV_1(S)+(4,5,6,8,10,12)$ (here ``+'' denotes concatenation of ordered sets). As in \NoLoops\ and \NoDoubles, we define $b_t(P',\OV_i(S))$ to be the number of ordered $W$ such that $\OV_i(S)+W=\OV_{i+1}(S')$ for some $t$-switching $S'$ that creates $P'$. 

Regarding complexity, after generating the pairing we can make an initial computation to locate all $O(d^2)$ multiple edges in time $O(dn)$. Similar to the argument in Section~\ref{sec:time}, there is no need to compute $f_t(P)$ for f-rejections. Each $b_t(P,\OV_i(S))$ can also be computed initially in time $O(dn)$ and updated in time $O(d^2)$. We can do this by additionally keeping track of the quantity $S_3(P)$ --- the number of simple 3-stars in $P$, which requires $O(dn+d^3)$ for initial computation and $O(d^2)$ time for updating after each $t$-switching. Then $b_t(P,\OV_0(S))$ is given by $S_3(P)$ and $b_t(P,\OV_1(S))$ can be computed as $S_3(P)-X$, where $X$ is number of bad choices of 3-stars $(4,5,6,8,10,12)$ due to vertex collision or forbidden edges present. Similar to the argument in Theorem~\ref{thm:complexity}, $X$ can be calculated in time $d^2$. Since $B_T=O(\log n+d^3/n)$, the total run time is bounded by $O(dn+d^3+d^5/n+d^2\log n)=O(dn+d^4)$ in this phase.

To complete the definition of the b-rejection scheme, we specify the following upper bound for $f_t(P)$, and lower bounds for $b_t(P,\OV_i)$, for $P$ containing exactly $j$ triple edges. These bounds are easy to verify with straightforward inclusion-exclusion arguments.
\begin{align*}
\overline{f}_{t}(j)&=12j M_1^3,\\
\LB_{t}(j;0)&=M_3\left(1-\frac{12B_Dd^2+18jd^2}{M_3}\right),\\
\LB_{t}(j;1)&=M_3\left(1-\frac{12B_Dd^2+18jd^2+16d^3+d^4}{M_3}\right).
\end{align*}
Now, after a uniformly random $t$-switching $S\sim(P',\OV(S))$ converting a pairing $P $ to $P'$ is selected, the switching is performed with probability $$\frac{f_t(P)}{\UB_t(j)}\frac{\LB_{t}(j-1;0)\LB_{t}(j-1;1)}{b_t(P',\OV_0(S))b_t(P',\OV_1(S))},$$ and is rejected with the remaining probability.

Finally, the {\em last phase} reduces the number of double edges. In {\tt REG}, this phase uses two types of switchings, type I and type II. They are drawn in Figures~\ref{f:I} and~\ref{f:II}. In a type I switching, along with the removal of the designated double edge, it is allowed to simultaneously create a new double edge, but no more than one. If no new double edge is created, the switching is said to be in class A, otherwise, it is in class B. See Figure~\ref{f:Ib} for an example of a type I, class B switching. A type II switching always deletes a designated double edge, and simultaneously creates exactly two double edges, and a type II switching is always in class B. In each switching step, for a pairing $P$ with $i$ double edges, {\tt REG} first chooses a switching type $\tau\in \{I,II\}$ with a specified distribution $\{\rho_I(i), \rho_{II}(i)\}$  over $\{I,II\}$, then uniformly selects a random type $\tau$ switching that can be performed on $P$ and obtains a pairing $P'$. An f-rejection may occur at this point. If the selected switching is of class $\alpha\in \{A,B\}$, {\tt REG} performs a b-rejection based on the number of class $\alpha$ switchings that can produce the resulting pairing $P'$. We refer the interested readers to~\cite{gao17}*{Sections 2, 5} for the rationale of the uses of different types of switchings and the classification of switchings into multiple classes. 

 The last phase runs as a Markov chain, occasionally increasing or not changing, but usually decreasing, the number of double edges. (The steps that do not increase the number of double edges are only chosen with very small probabilities.)  Once it reaches a pairing with no double edges, it outputs this pairing. In {\tt REG} the parameters $\{\rho_I(i), \rho_{II}(i)\}$ are chosen so that 
 \begin{itemize}
 \item[(i)] the expected number of times a switching $S\sim(P',\OV_2(S))$ appears in the algorithm after f-rejection depends only on the class of $S$ and the number of double edges in $P'$.
 \item[(ii)] the expected number of times a pairing $P$ is reached in {\tt REG} depends only on the number of double edges in $P$.
 \end{itemize}
 The critical property of b-rejection that is used in the derivation of the parameters $\{\rho_I(i), \rho_{II}(i)\}$ is property (iii) described below. For a pairing $P$ and a class $\alpha\in\{A,B\}$, let $S_\alpha(P)$ be the set of class $\alpha$ switchings that result in $P$. In {\tt REG},  b-rejection satisfies the following  
\begin{itemize}
    \item[(iii)] for all $P$ with $j$ double edges and all $\alpha\in\{A,B\}$  $$
    \sum_{S\in S_\alpha(P)}\mathbb{P}(S \; \text{is not rejected})=\LB_{\alpha}(j) $$
\end{itemize}
for some constants $\LB_{\alpha}(j)$ that were specified in~\cite{gao17}. We note here that as long as property (iii) is satisfied for the same set of constants $\LB_{\alpha}(j)$, we can replace b-rejection in {\tt REG} with any other rejection scheme and the modified algorithm that we obtain would still satisfy (i) and (ii).

 \begin{figure}[!tbp]
  \centering
  \begin{minipage}[b]{0.45\textwidth}
 \hbox{\centerline{\includegraphics[width=7cm]{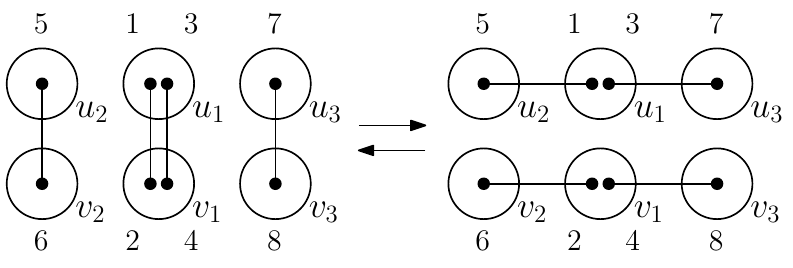}}}
\caption{Type I, Class A }
\label{f:I}
  \end{minipage}
  \hfill
  \begin{minipage}[b]{0.45\textwidth}
\hbox{\centerline{\includegraphics[width=7cm]{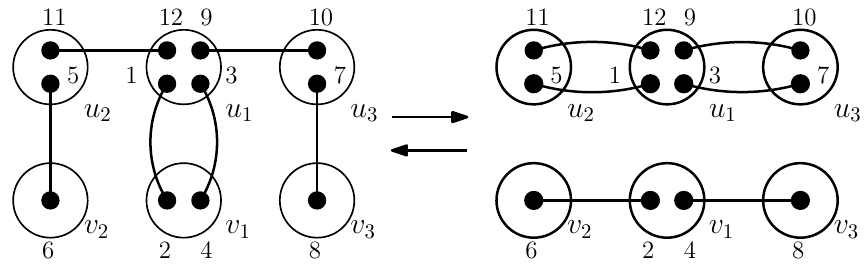}}}
\caption{Type II, Class B}
\label{f:II}
  \end{minipage}
\end{figure}

\begin{figure}[!tbp]
  \centering
   \hbox{\centerline{\includegraphics[width=7cm]{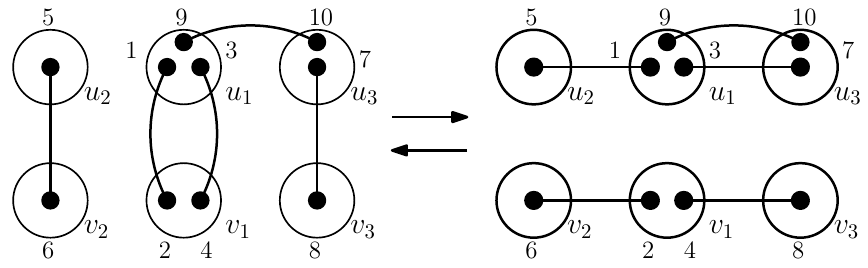}}}
\caption{Type I, Class B }
\label{f:Ib}
\end{figure}
Our new algorithm \reg\ uses same set of switchings as in~\cite{gao17}. The only nontrivial modifications are related to b-rejection, namely we obtain \reg\  by replacing b-rejection in {\tt REG} for class $\alpha$ switching with a corresponding version of incremental relaxation. Let $b_{\alpha}(P)$ be the number of class $\alpha$ switchings that produce $P$. The parameters $\LB_{\alpha}(j)$ are defined in~\cite{gao17} as certain uniform lower bounds for $b_{\alpha}(P)$, for pairings $P$ containing exactly $j$ double edges. Instead of computing $b_{\alpha}(P)$ as in~\cite{gao17}, we will now compute the quantity $$b_{\alpha}(P,\OV(S))=b_{\alpha}(P,\OV_0(S))b_{\alpha}(P,\OV_1(S)),$$ which depends on the switching $S$ that converts some pairing $P'$ to $P$. (Sets $\OV_0(S)$ and $\OV_1(S)$ are defined in the end of this section for each of the two classes of switchings.) In {\tt REG}, a graph is not b-rejected with probability $\LB_\alpha(j) / b_\alpha(P)$, while in \reg\ we set the probability of performing incremental relaxation without rejection to be $\LB_\alpha(j) / b_\alpha(P, \OV(S))$. It is straightforward to check that the constants $\LB_{\alpha}(j)$ are still lower bounds for $b_{\alpha}(P,\OV(S))$. In this context, the ideas in the proof of Lemma~\ref{lemma:unroot} can be used to show that 
$$
\sum_{S\in \S_\alpha(P)} \frac{1}{b(P, \OV_0(S))b(P, \OV_1(S))}=1$$
for all $P$ and $\alpha$, and so the relaxation scheme in \reg\ satisfies property (iii) with the same constants $\LB_{\alpha}(i)$. Hence, \reg\ also satisfies properties (i) and (ii) and so every simple pairing is generated with the same probability. Once again we only need to compute $b_\alpha(P, \OV_i(S))$ to run the last phase of the algorithm.

To complete the description of sequential relaxation we need to consider two different anchorings for $\alpha\in\{A,B\}$. For $\alpha=A$, only type I switchings can be in class A, and the type-I-class-A switchings are exactly the $d$-switchings defined in Section~\ref{sec:algorithm}. Thus, we will define $b_{A}(P,\OV_i(S))$
exactly the same as $b_d(G,\OV_i(S))$ in Section~\ref{sec:algorithm} and lower bound for $b_{A}(P,\OV(S))$ is defined to be (as in~\cite{gao17})
$$\LB_{A}(j)=M_2^2-M_2(d-1)(16j+3d^2+d+6).$$ The total run time with contributions from computing b-rejection probabilities for class A switchings is then $O(dn+d^4)$.

Next, consider $\alpha=B$. Every class-B switching $S$ can be identified with its image, that is with an ordered set of points $\OV(S)$, being a permutation of $\{1,\ldots,10\}$, such that $\{9,10\}$ is in a double edge, $9$ corresponds to vertex $u_1$ or $v_1$, and such that a 2-path containing $\{9,10\}$ has either one or two double edges (in the later case, the point $12$ belongs to the same vertex as $9$). 

 To be more precise,  assuming $S\sim(P,\OV(S))$ is a switching of class B, we define $\OV_0(S)=\emptyset$, $\OV_1(S)$ to be an ordered set of six points involved in a non-simple 2-path (ordered by natural order), and $\OV_2(S)=\OV(S)$. There are essentially four possible places where edge $\{9,10\}$ can be, all resulting in formally different sets $\OV_1(S)$ and $\OV_2(S)$, for example, if $9$ is in vertex $u_1$ and $10$ is in $u_3$, then $\OV_1(S)=(1,3,5,7,9,10)$ and $\OV_2(S)=\OV_1(S)+(2,4,6,8)$. Similar to Section~\ref{subsec:nodoubles}, for $i=0,1$ let $b_B(P,\OV_i)$ denote the number of ordered $W$ such that $\OV_i(S)+W=\OV_{i+1}(S')$ for some class B switching $S'$ that creates $P$. For a switching $S\sim(P,\OV(S))$  we set 
$$b_B(P,S)=b_B(P,\OV_0(S))b_B(P,\OV_1(S)).$$
A uniform lower bound $\LB_B(j)$ for $b_B(P,\OV(S))$ which depends only on $j$, the number of double edges in $P$, can be defined by
$$\LB_B(j)=\LB_B(j;0)\LB_B(j;1),$$ where
\begin{align*}
\LB_B(j;0)&=16j(d-2),\\
\LB_B(j;1)&=M_2\left(1-\frac{4jd+9d^2+3d^3}{M_2}\right).
\end{align*}
Note that the value $b_B(P,\OV_0(S))$ is always equal to $16j(d-2)$, as there are $4j$ possibilities to choose edge $\{9,10\}$ as one of the edges in a double edge $e$, four possibilities to label pair of 
$e$ different from $\{9,10\}$ (this pair can be labelled as $\{1,5\},\{2,6\},\{3,7\}$, or $\{4,8\}$) and for each such choice there are exactly $ d-2 $ choices for a second edge in a $2$-path involving $e$. For the value of $b_B(P,\OV_1(S))$, we use the same procedure as we used to calculate $b_{d}(G', \OV_1(S))$ in Theorem~\ref{thm:complexity}, so $b_B(P,S)$ can be calculated in time $O(d^2)$. It now follows from the proof in~\cite{gao17} that the total run time, including the contributions from computing b-rejection probabilities for class B switchings, is   $O(dn+d^4)$. 

\section{Power-law degree sequences}
\label{sec:powerlaw}

Our approach can be implemented to accelerate an existing algorithm for the uniform sampling of graphs with a degree sequence whose degree frequencies approximately follow a given power-law. The degree sequences being addressed can contain much larger degrees than permitted in the algorithms described so far in this paper. In the extended abstract of the present paper~\cite{agw20}, the authors presented an  algorithm INC-POWERLAW for this purpose, that uses exactly the same switchings as in~\cite{gao18} and claimed linear expected run-time. Unfortunately, there was a glitch in that algorithm since one step required super-linear run-time. The algorithm was repaired by
     Allendorf~\cite{allendorf2020} in consultation with the authors of the present paper, to maintain linear time, at the expense of introducing many more kinds of switching operations.
  
 The main difficulty with such degree sequences is that the multiplicities of edges between   vertices in a random pairing can be very large. The algorithm in~\cite{gao18} consists of two stages. In the first stage, multiple edges and loops of high multiplicities are switched away. By the end of  the first stage, the only remaining multiple edges are single loops, double edges or triple edges. The time complexity for the first stage is already only $o(n)$ in expectation (see Lemma 11 in~\cite{gao18}); this is quick because the expected number of edges involved in multiple edges is quite small.  The second stage contains three phases during which loops, triple edges and double edges in turn are removed using switchings. The most complicated phase is for the removal of the double edges.  This involves six different types of switchings.  

\powerlaw\ is identical to  the algorithm in~\cite{gao18} for the first stage. In the second stage, \powerlaw~uses the same switchings and rejection scheme as in \reg\ for the deletion of loops and triple edges. For the third stage (elimination of double edges), \powerlaw\ uses the same types of switchings as in~\cite{gao18}, and the modified version in~\cite{allendorf2020} uses 18 kinds of switchings. This phase uses incremental relaxation for b-rejection in the same way as described for \reg\ in Section~\ref{sec:regular}. We omit a detailed proof of  Theorem~\ref{thm:powerlaw} since the full story is given in~\cite{allendorf2020}.
 
 \begin{figure}[!tbp]
  \centering
  \begin{minipage}[b]{0.45\textwidth}
 \hbox{\centerline{\includegraphics[width=7cm]{switchingdoubleI}}}
\caption{Type I }
\label{f:1}
  \end{minipage}
  \hfill
  \begin{minipage}[b]{0.45\textwidth}
\hbox{\centerline{\includegraphics[width=7cm]{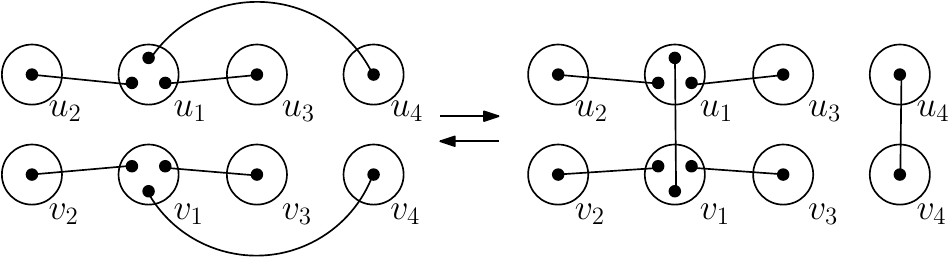}}}
\caption{Type III }
\label{f:3}
  \end{minipage}
\end{figure}

 
\section{Bipartite graphs}
\label{sec:bipartite}
With some minor modification our algorithm can be adjusted for generation of bipartite graphs with one part $X$ having degrees $\bfs=(s_1,\ldots,s_m)$ and the other   part $Y$ having degrees $\bft=(t_1,\ldots,t_n)$. 
Define 
\begin{align*}
M&=\sum_{i\in X} s_i=\sum_{j\in Y} t_j; \\
S_2&=\sum_{i\in X} s_i(s_i-1); \ \ T_2=\sum_{j\in Y} t_j(t_j-1).
\end{align*}

The algorithm \bipartite\ first uses the configuration model to generate a uniformly random pairing $P$ with bipartite degree sequence $(\bfs,\bft)$. The configuration model for a bipartite degree sequence is similar to the one for a general degree sequence, except that points in vertices of $X$ are restricted to be matched to points in vertices of $Y$. Let $\Phi(\bfs,\bft)$ denote the set of pairings with bipartite degree sequence $(\bfs,\bft)$, and $\Phi_0\subseteq \Phi(\bfs,\bft)$ be those containing at most $S_2T_2\slash M^2$ double edges and no other types of multiple edges.  An initial rejection is applied if $P\notin\Phi_0$.

The following lemma, which is based on Lemmas 2B and 3B$^\prime$ from~\cite{mckay90}, guarantees that the probability of an initial rejection is bounded away from $1$, provided $\Delta^4=O(M)$.

\begin{lemma} Let $P$ be a uniformly random pairing in $\Phi(\bfd)$.
There exists a constant $0<c<1$ such that $\pr(P\in\Phi_0)>c$ for all sufficiently large $n$.
\end{lemma}

To remove the double edges, Algorithm \bipartite\ uses the bipartite version of the $d$-switching operation in Section~\ref{sec:algorithm}, in which vertices $v_2,v_4,v_6$ are in $X$ and vertices $v_1,v_3,v_5$ are in $Y$.

We define $b_d(G',\OV(S))$ as before and we redefine
\begin{align*}
\LB_{d}(\bfm;0)&=T_2\left(1-\frac{4m_2\Delta}{T_2}\right),\\
\LB_{d}(\bfm;1)&=S_2\left(1-\frac{4m_2\Delta+4\Delta^2+ 3\Delta^3}{S_2}\right)
\end{align*}

Following a similar proof we have the following bipartite version of Lemma~\ref{lemma:d-bounds}.

\begin{lemma}\label{lemma:Bi-d-bounds}
Let $G'\in \G_{0,m_2}$ with $m_2\leq S_2T_2\slash M^2$. Then for any simple ordered 2-path $v_1v_2v_3$ in $G'$ we have
\begin{align*}
  &\LB_d(\bfm;0)\leq  b_d(G,\emptyset) \leq T_2\\  
  &\LB_d(\bfm;1)\leq  b_d(G,v_1v_2v_3) \leq S_2\\  
  & m_2M^4\left(1-\frac{8m_2+6\Delta^2+20\Delta}{M}\right)\leq  f_d(G) \leq \overline{f}_d(\bfm).
\end{align*}
\end{lemma}
Now we modify \NoDoubles\ in Section~\ref{sec:algorithm} by using the bipartite version of the $d$-switching operation, and the new definition of the parameters $\LB_d(\bfm ;i)$. Algorithm \bipartite~is given as follows.

\begin{algorithm}[H]
{\bf{procedure}}
{\bipartite $(\bfs,\bft)$:}\\
Generate a uniformly random $P\in\Phi(\bfs,\bfd)$. Initial reject if $P\notin \Phi_0$\;
Construct $G=G(P)$\;
\NoDoubles(G)\;
\end{algorithm}
\ss

Theorem~\ref{thm:bipartite} follows by 
 a proof almost identical to that of Theorem~\ref{thm:main}.

\bibliographystyle{plain}
\bibliography{biblio}

\end{document}